\documentclass[11pt,reqno]{amsart}

\usepackage{graphics}
\usepackage{amssymb, a4wide}
\usepackage{amsmath}
\usepackage{graphicx}
\usepackage{epsfig}
\DeclareMathOperator{\arcsinh}{arcsinh}

\newtheorem{proposition}{Proposition}
\newtheorem{theorem}{Theorem}
\newtheorem{example}{Example}
\newtheorem{remark}{Remark}
\newtheorem{lemma}{Lemma}
\newtheorem{corollary}{Corollary}
\newtheorem{definition}{Definition}
\DeclareMathOperator\arctanh{arctanh}

\title{Characterizations of Loxodromes on Rotational Surfaces in Euclidean 3--Space}

\author[F. Kahraman Aksoyak]{Ferda\u{g} Kahraman Aksoyak}
\address{K{\i}r\c{s}ehir Ahi Evran University,
Division of Elementary Mathematics Education,
40100, K{\i}r\c{s}ehir, Turkey}
\email{ferdag.aksoyak@ahievran.edu.tr}

\author[B. Bekta\c{s} Dem\.{i}rc\.i]{Burcu Bekta\c s Dem\.{i}rc\.i}
\address{Fatih Sultan Mehmet Vak{\i}f University, Hali\c{c} Campus, Faculty of Engineering,
Department of Software Engineering, 34445, Beyo\u{g}lu, istanbul, Turkey}
\email{bbektas@fsm.edu.tr}

\author[M. Babaarslan]{Murat Babaarslan}
\address{Yozgat Bozok University,
Department of Mathematics,
66100, Yozgat, Turkey}
\email{murat.babaarslan@bozok.edu.tr}



\date{\today}

\begin{document}

\parindent 0mm
\parskip 2mm

\maketitle

\begin{abstract}In this paper, we study on the characterizations of loxodromes on the rotational surfaces satisfying some special geometric properties such as having constant Gaussian curvature, flat and minimality in Euclidean 3-space. First, we give the parametrizations of loxodromes parametrized by arc-length parameter on any rotational surfaces in $\mathbb{E}^{3}$ and then, we calculate the curvature and the torsion of such loxodromes. 
Then, we give the parametrizations of loxodromes on rotational surfaces with constant Gaussian curvature. In particular, we prove that the loxodrome on the flat rotational surface is a general helix. Also, we investigate the loxodromes on the rotational surfaces with a constant ratio of principal curvatures (CRPC rotational surfaces). Moreover, we give the parametrizations of loxodromes on the minimal rotational surface which is a special case of CRPC rotational surfaces.  Then, we show that the loxodrome intersects the meridians of minimal rotational surface by the angle $\pi/{4}$ becomes an asymptotic curve. Finally, we give some visual examples to strengthen our main results via Wolfram Mathematica.
\end{abstract}

\section{Introduction} 
\label{Sec:1}

The geometry of curves and surfaces is not useful for mathematics but it also encompasses the techniques and ideas relevant to the studies in the engineering and the other fields of sciences. 

The curves such as loxodrome, logaritmic spiral, helix, geodesic curve and asymptotic curve in geometry are very popular in the engineering and the sciences. In 1537, a rhumb line (also called as a loxodrome) which makes a constant angle with all meridian of Earth’s surface was introduced by the Royal Cosmographer Pedro Nunes in the first time. Nunes made a clear distinction between the sailing along the shortest path (along a great-circle (geodesic)) and the sailing with a constant course along a rhumb line \cite{Leitao}. A great-circle course requires that changes continuously the bearing and it was an impossible task at first of ocean navigation. A navigator would usually approximate a great-circle by a series of rhumb lines \cite{Alexander}. The history of loxodromes goes back to the times when the voyagers first noticed that the Earth's surface is not flat and they had to take the curvature into account. In 1569, a major development  about loxodromes  was that Mercator projection of a loxodrome is a straight line \cite{Alexander}.

In 1905, Noble \cite{Noble} found the equations of loxodromes on rotational surface, spheroid and sphere in Euclidean 3-space. Also he showed that the stereographic projection of a loxodrome on sphere is a logarithmic spiral. The curvatures of loxodromes on rotational surfaces generated by the profile curve  $z=f(x)$ in $xz-$plane in Euclidean 3-space were computed by Kendall \cite{Kendall} in 1920. The arc-length of a loxodrome on a spheroid was computed by Petrovic \cite{Petrovic} in 2007. Also the arc-length of loxodrome on a sphere was given by Kos et. al. \cite{Kos} in 2009.

A curve whose tangent lines make a constant angle with a fixed direction is called as general helix. The helices on helicoid which is only ruled minimal surface are also geodesic curves. Similar as loxodromes, helices are important in navigation, too. For some interesting and magical applications of helices, we refer to \cite{Babaarslan, Deshmukha} and references therein.

A curve always tangents to an asymptotic direction of the surface is called as an asymptotic curve. These curves have been the subject of many studies in astronomy and in astrophysics (see \cite{Khalifa}).

The theory of surfaces is another useful and interesting field in the geometry. A minimal surface is a surface whose mean curvature is zero.  The first explicit examples of minimal surfaces are the plane, the catenoid (Euler, 1741) and the helicoid (Meusnier, 1776) (see \cite{Perez}). Minimal surfaces are used in many areas such as architecture, material science, aviation, ship manufacture, biology and so on (see \cite{Xu}). A Weingarten surface is a surface satisfying a relation between the principal curvatures. Also, CRPC surfaces as a special case of Weingarten surfaces are a natural generalization of minimal surfaces, too (see \cite{Wang}).

Although there are many studies about loxodromes on rotational surfaces, there is no enough results about the curvatures of loxodromes and the relations between loxodromes and geometric properties of surfaces. Therefore, in this article, we compute the curvatures and parametrizations of loxodromes for some special rotational surfaces such as flat, constant Gaussian curvature, Weingarten and minimal in Euclidean 3-space. Also, we investigate some important relationships between loxodromes, helices and asymptotic curves. Finally, we give some examples by using Wolfram Mathematica.

\section{Preliminaries}
\label{Sec:2}

In this section, we give some basic definitions and formulas related to curves and surfaces in Euclidean 3-space $\mathbb{E}^{3}$. For more details, we refer to \cite{O'Neill} and \cite{Kuhnel}.

The Euclidean inner product $\left\langle \cdot,\cdot\right\rangle$ in $\mathbb{E}^{3}$ allows us to define the length of vectors by the norm:
\begin{eqnarray}
\left\|x\right\| =\sqrt{\left\langle x,x\right\rangle}
\end{eqnarray}
as well as introducing the angle $\psi\in [0,\pi]$ between two vectors $x, y \neq 0$ by
\begin{eqnarray}
\label{angle}
\cos \psi = \frac{\langle x,y \rangle}{\left\|x\right \|  \left\|y\right\|}.
\end{eqnarray}

A regular parametrized curve is a continuously differentiable immersion $\alpha:I\rightarrow \mathbb{E}^3$, that is $\dot{\alpha}(t)\neq0$ holds everywhere, where $I\subset \mathbb{R}$ is a real interval.

The arc-length parameter of a curve $\alpha$ is introduced by
\begin{eqnarray}
s(t)=\int _{t_{0} }^{t}\left\|\dot{\alpha}(t)\right\|dt.
\end{eqnarray}
Also, $\alpha$ is a unit speed curve if $\|\alpha'(s)\|=1$ for all $s\in I\subset\mathbb{R}$.

Then, Frenet--Serret vectors of $\alpha$ are given by
\begin{displaymath}
\begin{array}{ccc}
T(t) & = & \frac{\dot{\alpha}(t)}{\left\|\dot{\alpha}(t)\right\|},\\\\
N(t) & = & B(t)\times T(t),\\\\
B(t) & = & \frac{\dot{\alpha}(t) \times \ddot{\alpha}(t)}{\left\| \dot{\alpha}(t) \times \ddot{\alpha}(t)\right\|},
\end{array}
\end{displaymath}

and the curvature and torsion of $\alpha$ is given by
\begin{eqnarray}
\label{curvature}
\begin{array}{ccc}
\kappa(t) & = & \frac{\left\| \dot{\alpha}(t) \times \ddot{\alpha}(t)\right\|}{\left\|\dot{\alpha}(t)\right\|^3},\\\\
\tau(t) & = & \frac{\det\left(\dot{\alpha}(t),\ddot{\alpha}(t),\dddot{\alpha}(t)\right)}{\left\| \dot{\alpha}(t) \times \ddot{\alpha}(t)\right\|^2}.
\end{array}
\end{eqnarray}

A unit speed curve $\alpha$ is called a general helix if the unit tangent vector makes a constant angle with a fixed unit vector, that is, $\left\langle T(s),u\right\rangle=\cos \gamma$ for some fixed angle $\gamma$. A space curve $\alpha$ is a general helix if and only if the ratio of $\kappa$ and $\tau$ is a constant. When both $\kappa$ and $\tau$ are constants, $\alpha$ is called a circular helix (see \cite{Deshmukha}).

A parametrized surface element (parametrization) is an immersion $\Phi:U\rightarrow \mathbb{E}^3$, where $U\subset \mathbb{R}^2$ is an open subset.

If $p$ is a point of orientable surface $M$, then for each tangent vector $v$ to $M$ at $p$, the shape operator of $M$ at $p$ is defined by
\begin{eqnarray}
\label{shapeop}
S_{p}(v)=-D_{v}N
\end{eqnarray}
where $D$ is covariant derivative and $N$ is a unit normal vector field on neighborhood of $p$ in $M$.

The first and second fundamental forms of $M$ are defined
\begin{eqnarray}
I(v,w)=\left\langle v,w\right\rangle, \, II(v,w)=\left\langle S(v),w\right\rangle
\end{eqnarray}
for all pairs of tangent vectors to $M$.

Let $v$ tangent vector to $M$ at $p$. Then $\kappa_{n}(v)=II(v,v)$ is called as the normal curvature of $M$ in the direction $v$. If the normal curvature of $v$ is zero then, $v$ is called as an asymptotic direction. A regular curve $\alpha:I\rightarrow M$ is an asymptotic curve provided its velocity vectors $\dot{\alpha}$ are asymptotic directions, that is, $\kappa_{n}(\dot{\alpha})=II(\dot{\alpha},\dot{\alpha})=0$.

For principal curvatures $\kappa_{1}, \kappa_{2}$ of $M$, the Gaussian curvature and mean curvature of $M$ are defined by
\begin{eqnarray}
K=\det(S)=\kappa_{1} \cdot \kappa_{2}
\end{eqnarray}
and 
\begin{eqnarray}
H=\frac{1}{2} \mathrm{tr}(S)=\frac{1}{2}(\kappa_{1}+\kappa_{2}),
\end{eqnarray}
respectively. It is well known that if $K=0$, then the surface is flat and if $H=0$, then the surface is minimal.

Let $R$ be a rotational surface in $\mathbb{E}^3$ given by
\begin{equation}
\label{rotsurf}
\Phi(s,\theta)=(f(s)\cos\theta, f(s)\sin\theta, g(s))
\end{equation}
whose the profile curve $\beta(s)=(f(s), g(s), 0)$ has an arc-length
parametrization, that is, $f'^2(s)+g'^2(s)=1$.
The unit normal vector of the rotational surface is%
\begin{equation}
\label{normal}
N(s,\theta )=\left( -g'(s)\cos \theta ,-g'(s)\sin \theta
,f'(s)\right).
\end{equation}
Using the equation  $f'^2(s)+g'^2(s)=1$, 
the principal curvatures of $R$ are calculated as
\begin{eqnarray}
\label{kappa_{1}}
\kappa_{1} & = & \frac{-f''}{\sqrt{1-f'^2}}, \\
\label{kappa_{2}}
\kappa_{2} & = & \frac{\sqrt{1-f'^2}}{f}.
\end{eqnarray}
Thus, the Gaussian and mean curvatures of $R$ in $\mathbb{E}^3$ 
are given by
\begin{align}
\label{Gauss}
K & =  -\frac{f''}{f}, \\
\label{mean}
H & = \frac{-f f^{\prime \prime }-f'^2+1}{2f\sqrt{1-f'^2}}.    
\end{align}
\begin{theorem}
\label{classminimal}
\cite{Pressley}
Any minimal rotational surface is either part of a plane or part of a
catenoid. If the rotational surface given by \eqref{rotsurf} is minimal, then the components of profile curve are obtained by%
\begin{equation}
\label{minimal1}
f(s)=\frac{1}{n}\sqrt{1+n^{2}\left( s+m\right) ^{2}}
\end{equation}%
and%
\begin{equation}
\label{minimal2}
g(s)=\pm \frac{1}{n}\arcsinh\left( n\left( s+m\right) \right) +r
\end{equation}%
where $r$ is a constant. Thus, we have 
$\displaystyle{f=\frac{1}{n}\cosh \left( n\left( g-r\right) \right)}$.
\end{theorem}

For the rotational surface $R$ with a constant Gaussian curvature $K_0$, the equation \eqref{Gauss} gives the differential equation $f''(s)=-K_0f(s)$. The solution of this differential equation according the choice of $K_0$ 
is given in \cite{Kuhnel} as follows:
\begin{theorem}
\label{classGauss}
\cite{Kuhnel}
Let $R$ be a rotational surface in $\mathbb{E}^3$ with constant Gaussian curvature $K_0$ given by \eqref{rotsurf}. Then, 
the function $f(s)$ is one of the followings:
\begin{itemize}
    \item [(i.)] for $K_{0}>0$,
    \begin{equation}
        \label{posGaus}
        f(s)=A\cos{(\sqrt{K_{0}}s)}+B\sin{(\sqrt{K_{0}}s)},
    \end{equation}
    
    \item [(ii.)] for $K_{0}=0$,
     \begin{equation}
        \label{zeroGaus}
       f(s)=As+B,\;\;\;|A|\leq 1,
    \end{equation}
    
    \item [(iii.)] for $K_{0}<0$,
    \begin{equation}
        \label{negGaus}
     f(s)=A\cosh{(\sqrt{-K_{0}}s)}+B\sinh{(\sqrt{-K_{0}}s)}
    \end{equation}
\end{itemize}
where $A$ and $B$ are constants. 
\end{theorem}

Throughout the article, we denote the derivatives with respect to $s$ and $t$ by " $'$ " and " $\dot{}$ ", respectively.

\section{The Curvatures of Loxodromes on Rotational Surfaces }
\label{Sec:3}
In this section, we recall the definition and the parametrization of a loxodrome on a rotational surface in Euclidean 3--space. 
Then, we compute the curvature and torsion of such loxodromes.  

\begin{definition}
A loxodrome is a special curve on a rotational surface which meets the meridians at a constant angle.
\end{definition}

In \cite{Noble}, C. A. Noble found the explicit parametrization of any loxodrome on any rotational surface of $\mathbb{E}^3$.
On the other hand, we study the loxodrome on the rotational surface in $\mathbb{E}^3$ in case that the loxodrome and the profile curve of the rotational surface are parametrized by arc--length parameter.
Thus, we give the following statement about the parametrization of such loxodromes with its proof.  

\begin{theorem}
\label{thmloxodrome}
Let $R$ be a rotational surface in $\mathbb{E}^3$ given by \eqref{rotsurf}. Then, a parametrization of a loxodrome on $R$ is given by 
\begin{equation}
\label{lox}
    \alpha(t)=(f(s(t))\cos{\theta(t)}, f(s(t))\sin{\theta(t)}, g(s(t)))
\end{equation}
where $s(t)=(\cos{\psi})t+c$ and 
$\displaystyle{\theta(t)=\varepsilon\sin{\psi}\int^t_{t_0}
\frac{d\xi}{f(s(\xi))}}$
for constants $c, t_0$, $\varepsilon=\pm 1$ and $\psi\in [0,\pi]$.
\end{theorem}

\begin{proof}
Assume that $R$ is a rotational surface in $\mathbb{E}^3$ given by \eqref{rotsurf}. Any meridian of $R$, 
that is a curve $\theta=\theta_0$ for a constant $\theta_0$, 
is given by 
$\gamma_\theta(s)= (f(s)\cos{\theta}, f(s)\sin{\theta}, g(s))$. 
Taking the derivative of $\gamma_\theta$ with respect to $s$, 
we get 
\begin{equation}
\label{meridiander}
 \gamma'_\theta(s)= (f'(s)\cos{\theta}, f'(s)\sin{\theta}, g'(s)).
 \end{equation}
Since the profile curve of $R$ has arc--length parameter $s$,
$||\gamma'_\theta||=1$. 
On the other side, we can express the general form of a loxodrome $\alpha:I\subset\mathbb{R}\rightarrow R\subset\mathbb{E}^3$ as follows:
\begin{equation}
\label{loxodrome}
    \alpha(t)=(f(s(t))\cos{\theta(t)}, f(s(t))\sin{\theta(t)}, g(s(t))).
\end{equation}
By a direct calculation, we get 
\begin{equation}
\label{loxoderv}
    \dot{\alpha}(t)
    =
    (f'(s(t))\dot{s}\cos{\theta(t)}
    -f(s(t))\sin{\theta(t)}\dot{\theta}(t), 
    f'(s(t))\dot{s}\sin{\theta(t)}
    +f(s(t))\cos{\theta(t)}\dot{\theta}(t), g'(s(t))\dot{s}).
\end{equation}
Since $\alpha$ is parametrized by arc--length parameter $t$, 
that is $||\dot{\alpha}||=1$, we get 
the following equation
\begin{equation}
\label{loxarc}
\dot{s}^2(t)+f^2(s(t))\dot{\theta}^2(t)=1.
\end{equation}
%Without loss of generality, we take the starting point of the %loxodrome and the meridian as 
%$\alpha(t_0)=\gamma_\theta(s_0)=(f(s_0), 0, g(s_0))$.
%Thus, this gives the initial conditions  for the functions $s(t)$ %and $\theta(t)$ as $s(t_0)=s_0$ and $\theta(t_0)=\theta_0$. 
Using the equations \eqref{meridiander} and \eqref{loxoderv}  in \eqref{angle}, we get
\begin{equation}
    \cos{\psi}=\langle\dot{\alpha}(t),\gamma'_\theta(s)\rangle
    =\dot{s}(t)
\end{equation}
for a constant $\psi\in [0,\pi]$.
Thus, we find $s(t)=(\cos{\psi})t+c$ for any constant $c$. 
Substituting $\dot{s}(t)$ in the equation \eqref{loxarc}, we get 
the function $\theta(t)$ as given. 
\end{proof}
Then, we calculate the curvature and torsion of the loxodrome $\alpha$ on $R$ in $\mathbb{E}^3$ parametrized by \eqref{lox}.

\begin{proposition}
Let $R$ be a rotational surface in $\mathbb{E}^3$ and 
$\alpha$ be a loxodrome on $R$ given by \eqref{rotsurf} and \eqref{lox}, respectively. The curvature and the torsion of the loxodrome $\alpha$ are given by
\begin{align}
\label{kaplox} 
\kappa&= \sqrt{\frac{a^4f''^2}{1-f'^2}
+\frac{b^4+a^2b^2f'^2}{f^2}-\frac{2a^2b^2f''}{f}},\\
\label{torlox}
\tau=-\frac{ab}{\kappa^2f^3(1-f'^2)^{3/2}}
&(-b^2(1-f'^2)^2(b^2+a^2f'^2)
+a^2ff''(1-f'^2)(3b^2+(a^2-2b^2)f'^2)\notag\\
&+a^2f^2f''^2(-2a^2+b^2+3a^2f'^2)
-a^4f^3f''^3
+a^2f^2f'f'''(1-f'^2))
\end{align}
where $f$ denotes a smooth function $f(s(t))$, 
$a=\cos{\psi}$ and $b=\varepsilon\sin{\psi}$ for constants $\psi\in [0,\pi]$ and $\varepsilon=\pm 1$. 
\end{proposition}

\begin{proof}
Assume that $\alpha(t)$ is a loxodrome on $R$ in $\mathbb{E}^3$ 
given by \eqref{lox}. For a simplicity, we define $a=\cos\psi$ and $b=\varepsilon\sin\psi$. 
Then, we calculate the first and second derivatives of $\alpha(t)$ as follows:
\begin{align}
\label{firstder}
    \dot{\alpha}(t)&=(af'\cos{\theta}-b\sin{\theta}, af'\sin{\theta}+b\cos{\theta}, ag'),\\
\label{secder}    
    \ddot{\alpha}(t)&=
    \left(\left(-\frac{b^2}{f}+a^2f''\right)\cos{\theta}
    -\frac{abf'}{f}\sin{\theta}, 
    \left(-\frac{b^2}{f}+a^2f''\right)\sin{\theta}
    +\frac{abf'}{f}\cos{\theta}, a^2g''\right).
\end{align}
From $g'(s)=\pm\sqrt{1-f'^2(s)}$, 
we have $g''(s)=\mp\frac{f'(s)f''(s)}{\sqrt{1-f'^2(s)}}$.
Using these equations in the first equation of \eqref{curvature}, 
we get $\kappa$ as the equation 
\eqref{kaplox}. 
Also, from the equation \eqref{secder}, we compute 
\begin{align}
\label{thirder}
\dddot{\alpha}(t)=
   \Big(\frac{b}{f^2}(b^2+a^2f'^2-2ff'')\sin{\theta}
    +a^3f'''\cos{\theta},
    -\frac{b}{f^2}(b^2+a^2f'^2-2ff'')\cos{\theta}
    &+a^3f'''\sin{\theta},\notag\\
    &a^3g'''\Big).
\end{align}
Since $\alpha(t)$ has arc-length parametrization, 
the curvature of $\alpha$ is also given by $\kappa=||\dot{\alpha}\times\ddot{\alpha}||$. 
Then, 
$\tau=\frac{\mbox{det}(\dot{\alpha},\ddot{\alpha},\dddot{\alpha})}
{\kappa^2}$.
Moreover, we have 
$g'''(s)=\mp\frac{f'(s)f'''(s)(1-f'^2(s))+f''^2(s)}
{(1-f'^2(s))^{3/2}}$.
By a direct calculation, using the equations \eqref{firstder}, \eqref{secder} and \eqref{thirder} in the second equation of \eqref{curvature}, we find the torsion $\tau$ as the equation \eqref{torlox}.
\end{proof}

Now, we examine some special values for the angle $\psi$ 
to make characterization for the loxodrome $\alpha$ on the rotational surface $R$ in $\mathbb{E}^3$. 

\textbf{Case 1.}
For $\psi=0$, that is, $a=1$ and $b=0$. In this case,
the loxodrome coincides the meridian curve of 
the rotational surface $R$ in $\mathbb{E}^3$ and
the equations \eqref{kaplox} and \eqref{torlox} 
also give $\kappa=\pm\frac{f''}{\sqrt{1-f'^2}}$ and $\tau=0$ 
which are the curvature and torsion of the meridian curve $\gamma_\theta$. It is well known that all meridian curves are geodesics on the rotational surface in $\mathbb{E}^3$. 
Thus, the loxodrome becomes a geodesic of the rotational surface (see \cite{O'Neill}).

\textbf{Case 2.}
For $\psi=\frac{\pi}{2}$, that is, $a=0$ and $b=\varepsilon$,
the equations \eqref{kaplox} and \eqref{torlox} 
implies $\kappa=\pm\frac{1}{f(c)}$ and $\tau=0$
for a constant $c$. 
Then, the loxodrome becomes a parallel 
curve of the rotational surface $R$ in $\mathbb{E}^3$, 
that is, it is a circle.

\textbf{Case 3.}
For $\psi=\frac{\pi}{4}$, that is, $a^2=b^2$. The equations \eqref{kaplox} and \eqref{torlox} give the curvature and torsion of $\alpha$ as follows:
 \begin{align}
     \label{kaploxeq}
 \kappa&= \frac{1}{2}\sqrt{\frac{f''^2}{1-f'^2}
+\frac{1+f'^2}{f^2}-\frac{2f''}{f}},\\
    \label{torloxeq}
\tau=\frac{\varepsilon}{8\kappa^2f^3(1-f'^2)^{3/2}}
&((1-f'^2)^2(1+f'^2)
-ff''(1-f'^2)(3-f'^2)+f^2f''^2(1-3f'^2)\notag\\
&+f^3f''^3-2f^2f'f'''(1-f'^2))
\end{align}
where 
$f=f(s(t))$ is a smooth function for $s(t)=\frac{1}{\sqrt{2}}t+c$. 

\section{The characterizations of Loxodromes on Rotational Surfaces}
\label{Sec:4}
In this section, we investigate the loxodromes on rotational surfaces in $\mathbb{E}^3$ putting some conditions for rotational surface. First, we examine the loxodrome $\alpha$ on the rotational surface $R$ in $\mathbb{E}^3$ with constant Gaussian curvature $K_0$. 
\begin{lemma}
Let $R$ be a rotational surface in $\mathbb{E}^3$ with a constant Gaussian curvature $K_0$ and $\alpha$ be a loxodrome on $R$ given by \eqref{rotsurf} and \eqref{lox}, respectively. Then,
the curvature and the torsion of $\alpha$ are given by
\begin{align}
     \label{kaploxcg}
 \kappa= \sqrt{2a^2b^2K_0+\frac{a^4K_0^2f^2}{1-f'^2}
+\frac{b^4+a^2b^2f'^2}{f^2}},
\end{align}
\begin{align}
    \label{torloxcg}
\tau=-\frac{ab}{\kappa^2f^3(1-f'^2)^{3/2}}
&(a^4K_0^3f^6-b^2(1-f'^2)^2(b^2+a^2f'^2)
+a^2K_0^2f^4(-2a^2+b^2+3a^2f'^2)\notag\\
&-a^2K_0f^2(1-f'^2)(3b^2+(1+a^2-2b^2)f'^2))
\end{align}
where $K_0$ is an arbitrary constant.
\end{lemma}

\begin{proof}
Suppose that $R$ is a rotational surface in $\mathbb{E}^3$ with 
a constant Gaussian curvature $K_0$. 
From the equation \eqref{Gauss}, we get $f''(s)=-K_0f(s)$.
Considering this in the equations \eqref{kaplox}
and \eqref{torlox}, we get the curvature and torsion of $\alpha$ on 
such a rotational surface $R$ given by \eqref{kaploxcg} and \eqref{torloxcg}, respectively. 
\end{proof}

\begin{theorem}
Let $R$ be a rotational surface in $\mathbb{E}^3$  and $\alpha$ be a loxodrome on $R$ given by \eqref{rotsurf} and \eqref{lox},
respectively. If $R$ is a flat rotational surface, then the loxodrome $\alpha$ on $R$ becomes a general helix. 
\end{theorem}

\begin{proof}
Assume that $R$ is a flat rotational surface in $\mathbb{E}^3$ and $\alpha$ is a loxodrome on $R$.
For $K_0=0$, the equations \eqref{kaploxcg} and \eqref{torloxcg}
give
\begin{align}
     \kappa &=\pm\frac{b}{f}\sqrt{b^2+a^2f'^2},\\
    \tau&=\frac{ab}{f}\sqrt{1-f'^2}.   
\end{align}
On the other hand, the equation \eqref{zeroGaus} implies 
$f'(s)=A$ for a constant $|A|\leq 1$. 
Thus, it can be seen that 
the ratio of $\kappa$ and $\tau$ is a nonzero constant. 
Therefore, we get that the loxodrome $\alpha$ on $R$ is a general helix.
\end{proof}

Then, we give the explicit parametrization of such loxodromes by the following theorem. 
\begin{theorem}
\label{paraflat}
A loxodrome $\alpha$ on a flat rotational surface $R$ in $\mathbb{E}^3$ can be parametrized by
\begin{equation}
\label{loxzero}
    \alpha(t)=\left(
    \begin{array}{c}
    (aA t+B_0)\cos{\theta(t)},\\
    (aA t+B_0)\sin{\theta(t)},\\
    \pm a\sqrt{1-A^2}t+A_0
    \end{array}
    \right)
\end{equation}
where the function $\theta(t)$ is given by one of the followings:
\begin{itemize}
    \item [(i.)] $\displaystyle{\theta(t)
    =\frac{\varepsilon\tan{\psi}}{A}\ln{|aAt+A_0|}}+\theta_0$ 
    for $A\neq 0$,
    
    \item [(ii.)] $\displaystyle{\theta(t)
    =\frac{b}{B_0}t}+\theta_0$ for $A=0$ and $B_0\neq 0$
\end{itemize}
and $a=\cos{\psi}, b=\varepsilon\sin{\psi}$, 
$|A|\leq 1$, $A_0, B_0$ and $\theta_0$ are constants. 
Also, the loxodrome $\alpha$ is a circle on the plane which is orthogonal to the axis of rotation for $|A|=1$; 
it is a general helix on the circular cone for $0<|A|<1$ 
and it is a circular helix on the cylinder for $A=0$. 
\end{theorem}

\begin{proof}
Assume that $R$ is a flat rotational surface in $\mathbb{E}^3$, that is $K_0=0$. Then, Theorem \ref{classGauss} implies that the function $f(s)$ is given by \eqref{zeroGaus}. 
Since $f'^2(s)+g'^2(s)=1$ and $s(t)=at+c$, we find 
$g(s(t))=\pm a\sqrt{1-A^2}t+A_0$ 
for constants $A_0, A$ and $a=\cos{\psi}$.
Taking into account the function $f(s)$, 
we calculate $\theta(t)$ in Theorem \ref{thmloxodrome} 
with respect to the cases of constant $A$. 

Moreover, we know that if $A=0$, $R$ is a cylinder of radius $B$;
$R$ is a plane which is orthogonal to the axis of rotation if $|A|=1$ and $R$ is a circular cone if $0<|A|<1$, 
(see \cite{Kuhnel}). Thus, we get the desired result. 
\end{proof}

Now, we give an example of a loxodrome on a flat rotational surface in $\mathbb{E}^3$.

\begin{example}
Let take $\psi=\frac{\pi}{3}$ and $\varepsilon=1$. Then, we find $\theta(t)=2\sqrt{3}\ln{\frac{t}{4}}$ for $t>0$.
Thus, the parametrization of loxodrome (general helix) is
\[\alpha(t)=\left(\frac{1}{4}t\cos\left(2\sqrt{3}\ln{\frac{t}{4}}
\right), \frac{1}{4}t\sin\left(2\sqrt{3}\ln{\frac{t}{4}}\right), \frac{\sqrt{3}}{4}t\right).\]
The graphs of flat rotational surface (circular cone) $R$, loxodrome (general helix) and meridian ($\theta=0$) can be drawn by using Mathematica as follows:
\begin{figure}[htbp]
\includegraphics[height=90mm]{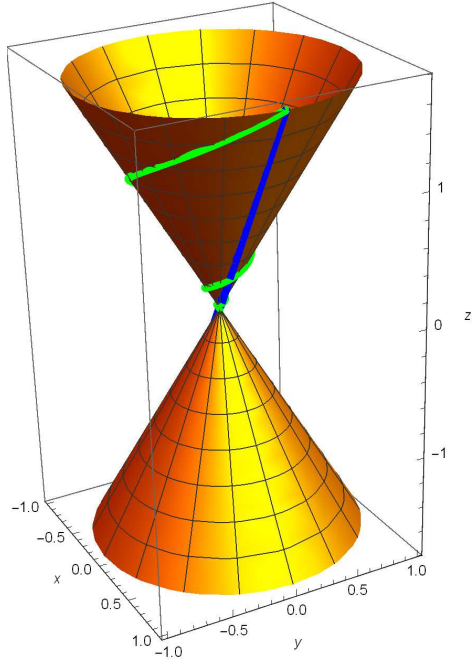}
\caption {Flat rotational surface (circular cone), loxodrome (general helix) (green), meridian (blue).}
\label{fig1}
\end{figure}
\end{example}

\newpage
\begin{theorem}
\label{paraposGauss}
A loxodrome $\alpha$ on a rotational surface $R$ in $\mathbb{E}^3$ with a positive constant Gaussian curvature $K_0$ can be parametrized by 
\begin{equation}
\label{loxpos}
    \alpha(t)=\left(
    \begin{array}{c}
    (A\cos{(a\sqrt{K_0}t+a_0)}+B\sin{(a\sqrt{K_0}t+a_0)})
    \cos{\theta(t)},\\
    (A\cos{(a\sqrt{K_0}t+a_0)}+B\sin{(a\sqrt{K_0}t+a_0)})
    \sin{\theta(t)},\\
    \displaystyle{\pm\int_{0}^{at+c}
    \sqrt{1-K_0(A\sin{(\sqrt{K_0}\xi)}-B\cos{(\sqrt{K_0}\xi)})^2}
    d\xi}
    \end{array}
    \right)
\end{equation}
where the function $\theta(t)$ is given by one of the followings:
\begin{itemize}
    \item [(i.)]
    $\displaystyle{\theta(t)
    =\frac{\varepsilon\tan{\psi}}{B\sqrt{K_0}}
    \ln{\left|{\tan\left({\frac{1}{2}(a\sqrt{K_0}t+a_0)}\right)}
    \right|+\theta_0}}$ for $A=0$ and $B\neq 0$,

    \item [(ii.)]
    $\displaystyle{\theta(t)
    =\frac{2\varepsilon\tan{\psi}}{\sqrt{K_0(A^2+B^2)}}
    \arctanh{\left(\frac{1}{\sqrt{A^2+B^2}}
    \left(A\tan{\left(\frac{1}{2}(a\sqrt{K_0}t+a_0)\right)}-B\right)\right)}+\theta_0}$ for $A\neq 0$
\end{itemize}
and $A, B, a_0, \theta_0$, $a=\cos{\psi}$ are constants.
\end{theorem}
\begin{proof}
Assume that the rotational surface $R$ in $\mathbb{E}^3$ has a positive constant Gaussian curvature $K_0$. 
Then, Theorem \ref{classGauss} implies that the function $f(s)$ is given by \eqref{posGaus}. 
Using $g'^2(s)=\pm\sqrt{1-f'^2(s)}$ and $s(t)=at+c$, 
we find the function
$g(s(t))$.  
Calculating $\theta(t)$ in Theorem \ref{thmloxodrome}  
for the function $f$ with respect to the cases of constants $A$ and $B$, we get desired result.  
\end{proof}

\begin{theorem}
\label{paranegGauss}
A loxodrome $\alpha$ on a rotational surface $R$ in $\mathbb{E}^3$ with a negative constant Gaussian curvature $K_0$ can be parametrized by  
\begin{equation}
\label{loxneg}
    \alpha(t)=\left(
    \begin{array}{c}
    (A\cosh{(a\sqrt{-K_0}t+a_0)}+B\sinh{(a\sqrt{-K_0}t+a_0)})
    \cos{\theta(t)},\\
    (A\cosh{(a\sqrt{-K_0}t+a_0)}+B\sinh{(a\sqrt{-K_0}t+a_0)})
    \sin{\theta(t)},\\
    \displaystyle{\pm\int_{0}^{at+c}
\sqrt{1+K_0(A\sinh{(\sqrt{-K_0}\xi)}
+B\cosh{(\sqrt{-K_0}\xi)})^2}d\xi}
    \end{array}
    \right)
\end{equation}
where the function $\theta(t)$ is given by one of the followings:
\begin{itemize}
    \item [(i.)] 
    $\displaystyle{\theta(t)
    =\frac{\varepsilon\tan{\psi}}{B\sqrt{-K_0}}
    \ln{\left|\tanh{\left(\frac{1}{2}(a\sqrt{-K_0}t+a_0)\right)}
    \right|}
    +\theta_0}$ 
    for $A=0$ and $B\neq 0$,
    
    \item [(ii.)] 
    $\displaystyle{\theta(t)
    =\frac{2\varepsilon\tan{\psi}}{\sqrt{K_0(B^2-A^2)}}
    \arctan{\left(\frac{1}{\sqrt{A^2-B^2}}
    \left(A\tanh{\left(\frac{1}{2}(a\sqrt{-K_0}t+a_0)\right)}
    +B\right)
    \right)}+\theta_0}$
    for $A\neq 0$ and $B^2<A^2$,
    
    \item [(iii.)]
     $\displaystyle{\theta(t)
    =\frac{2\varepsilon\tan{\psi}}{\sqrt{K_0(A^2-B^2)}}
    \arctanh{\left(\frac{1}{\sqrt{B^2-A^2}}
    \left(A\tanh{\left(\frac{1}{2}(a\sqrt{-K_0}t+a_0)\right)}
    +B\right)
    \right)}+\theta_0}$
    for $B\neq 0$ and $A^2<B^2$,
    
    \item [(iv.)]
    $\displaystyle{\theta(t)
    =-\frac{\varepsilon\tan{\psi}}{A\sqrt{-K_0}}
    e^{-(a\sqrt{-K_0}t+a_0)}+\theta_0}$ for $A=B\neq 0$
\end{itemize}
and $A, B, a_0, \theta_0$, $a=\cos{\psi}$ are constants.
\end{theorem}

\begin{proof}
Assume that a rotational surface $R$ in $\mathbb{E}^3$ has a negative constant Gaussian curvature $K_0$. 
Then, Theorem \ref{classGauss} implies that the function $f(s)$ is given by \eqref{negGaus}. 
From $g'(s)=\pm{\sqrt{1-f'^2(s)}}$ and $s(t)=at+c$, we find
the function $g(s(t))$.
Calculating $\theta(t)$ in Theorem \ref{thmloxodrome}  
for the function $f$ with respect to the cases of constants $A$ and $B$, we get desired result.  
\end{proof}

Moreover, we know that the rotational surface $R$ with a negative constant Gaussian curvature $K_0$ is 
so called conic type for $B^2>A^2$ and it is called hyperboloid
type for $B^2<A^2$, (see \cite{Kuhnel}).

Now, we give some examples of loxodromes on rotational surfaces with constant Gaussian curvature.
\begin{example}
For $K_{0}>0$ and $B=0$ in \eqref{posGaus}, we have 
\begin{equation}
f\left( s\right) =A\cos \left( \sqrt{K_{0}}s\right) \text{ and }g\left(
s\right) =\int_{0}^{s}\sqrt{1-A^{2}K_{0}\sin ^{2}(\sqrt{K_{0}}\xi) 
}d\xi.
\end{equation}
Here, $g(s)$ is the elliptic integral of second kind. 
Then if we take as $A^{2}K_{0}=1$ we get sphere, in case $0<A^{2}K_{0}<1,$
we have an elongated sphere and while $A^{2}K_{0}>1$ we obtain a oblate sphere. From Theorem \ref{paraposGauss}, 
the parametrization the loxodrome on sphere, elongated sphere or oblate sphere is found as:%
\begin{equation*}
\alpha \left( t\right) =\left( 
\begin{array}{c}
A\cos \left( a\sqrt{K_{0}}t\right)\cos{\theta(t)} , \\ 
A\cos \left( a\sqrt{K_{0}}t\right) \sin{\theta(t)}, \\ 
\displaystyle{\pm\int_{0}^{at+c} \sqrt{1-A^{2}K_{0}\sin ^{2}(\sqrt{K_{0}}\xi) }d\xi}%
\end{array}%
\right)
\end{equation*}
where 
$\theta(t)=
\frac{\varepsilon\tan{\psi}}{A\sqrt{K_0}}
\ln{|\sec{(a\sqrt{K_0}t+a_0)}+\tan{(a\sqrt{K_0}t+a_0)}|}
$.
%${\theta(t)
%    =\frac{2\varepsilon\tan{\psi}}{\sqrt{K_0A^2}}
%    \arctanh{
%    \left(\tan{\left(\frac{1}{2}(a\sqrt{K_0}t+a_0)\right)}\right)}}$
The curvature and torsion of the above curve can be easily computed by using
\eqref{kaploxcg} and \eqref{torloxcg} for $f\left( s\right) =A\cos \left( \sqrt{K_{0}}s\right) ,$ $%
s(t)=at+c.$

For $A=\frac{\sqrt{2}}{2}$, $K_{0}=1$, $\psi=\frac{\pi}{4}$ and $\varepsilon=1$, the graphs of elongated sphere $R$, loxodrome and meridian ($\theta=2\pi$) can be drawn by using Mathematica as follows:

\begin{figure}[htbp]
\includegraphics[height=90mm]{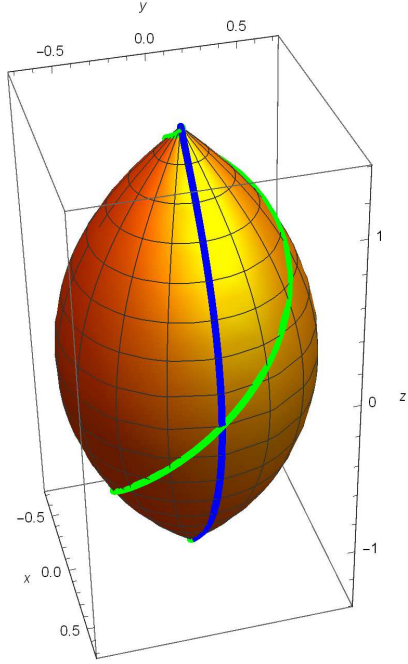}
\caption {Elongated sphere, loxodrome (green), meridian (blue).}
\label{fig3}
\end{figure}
\end{example}

\begin{example}
If we take as $A=B$ and $K_{0}=-1$ in \eqref{negGaus}, for 
\begin{equation*}
f\left( s\right) =Ae^{s}\text{ and }g\left( s\right) =\int_{0}^{s}\sqrt{%
1-A^{2}e^{2\xi}}d\xi,
\end{equation*}%
we get a rotational surface in $\mathbb{E}^3$ which is called Beltrami surface. 
From Theorem \ref{paranegGauss},
the parametrization of the loxodrome on the Beltrami surface is given by%
\begin{equation*}
\alpha \left( t\right) =\left( 
\begin{array}{c}
Ae^{at+c}\cos \left( -\frac{tan{\psi}}{A}%
\left(Ce^{-at}+D\right)\right), \\
Ae^{at+c}\sin \left( -\frac{tan{\psi}}{A}\left(Ce^{-at}+D\right)\right),\\
\displaystyle{\pm\int_{0}^{at+c}\sqrt{%
1-A^{2}e^{2\xi}}d\xi}
\end{array}
\right). 
\end{equation*} 
The curvature and the torsion of $\alpha$ are calculated as
\begin{equation}
    \kappa=\frac{\sqrt{b^4-b^2f^2+a^2f^4}}{f\sqrt{1-f^2}}
\end{equation}
and 
\begin{equation}
\tau= \frac{ab^3(b^2-2f^2+f^4)}{f\sqrt{1-f^2}(b^4-b^2f^2+a^4f^4)},
\end{equation}
where $f(s)=Ae^s$ and $s(t)=at+c$.

For $A=1$ and $\psi=\frac{\pi}{6}$, the graphs of Beltrami surface $R$, loxodrome and meridian ($\theta=3\pi/2$) can be drawn by using Mathematica as follows:
\begin{figure}[htbp]
\includegraphics[height=80mm]{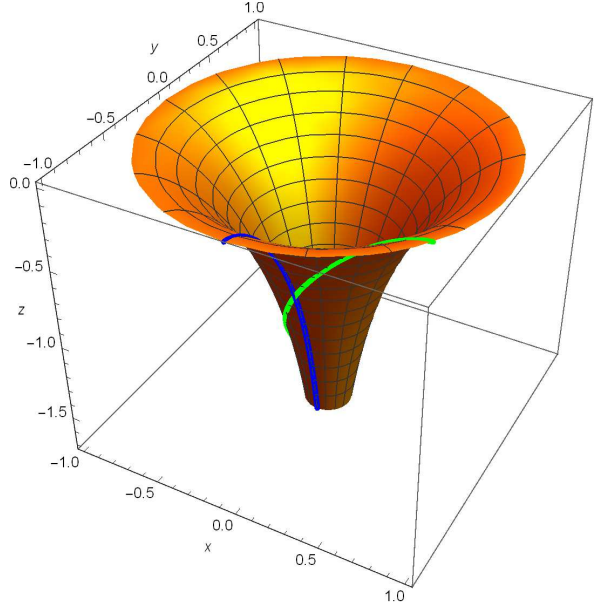}
\caption {Beltrami surface, loxodrome (green), meridian (blue).}
\label{fig2}
\end{figure}
\end{example}

Now, we investigate the loxodromes on CRPC rotational surfaces, 
that is $\kappa_1=k\kappa_2$ with a constant $k\neq 0$.
For $k=0$, the rotational surface $R$ becomes flat in which case
was examined before. 

\begin{lemma}
Let $R$ be a CRPC rotational surface such that $\kappa _{1}=k\kappa _{2}$ 
with a nonzero constant $k$ in 
$\mathbb{E}^{3}$ and $\alpha $ be a loxodrome on $R$. Then, the curvatures of $%
\alpha $ are as:%
\begin{equation}
\label{kappaw}
\kappa =\pm\frac{\sqrt{b^2+d^2a^2(a^2k^2+2kb^2-b^2)f^{2k}}}{f}
\end{equation}%
and 
\begin{equation}
\label{torw}
\tau=\pm \frac{abd}{\kappa^2f^{3-k}}
(b^2+a^2k^2
-d^2(a^2b^2+(a^4-2a^2b^2-a^2)k+(2a^2-3a^4)k^2+a^4k^3)f^{2k})
\end{equation}%
where $f$ denotes the function $f(s(t))$ and $d$ is a positive constant.
\end{lemma}

\begin{proof}
We assume that the principal curvatures of $R$ hold 
$\kappa _{1}=k\kappa_{2}$. 
Then by using \eqref{kappa_{1}} and \eqref{kappa_{2}},
we get the following differential equation%
\begin{equation}
\label{difeqwc}
f(s) f''(s) =k(f'^{2}(s)-1).
\end{equation}%
From \eqref{difeqwc}, we also obtain the following equation
\begin{equation}
\label{difeqwc1}
f^{\prime }(s) =\pm \sqrt{1-d^2f^{2k}(s)}.
\end{equation}%
By considering \eqref{difeqwc} and \eqref{difeqwc1} with \eqref{kaplox} and \eqref{torlox}, 
we obtain the curvatures of loxodrome $\alpha $ given
by \eqref{kappaw} and \eqref{torw}.
\end{proof}
\begin{remark}
In \cite{Kuhnel}, we know that the rotational surface $R$ in $\mathbb{E}^3$ 
becomes some special surfaces according to the choice of $k$ given as follows.
\begin{itemize}
    \item [(i.)] For $k=1$, $R$ is a umbilic rotational surface, that is, it is a unit sphere.
    
    \item [(ii.)] For $k=-1$, $R$ is a minimal rotational surface, that is, it is a catenoid.

    \item [(iii.)] For $k=-\frac{1}{2}$, 
    $R$ is a Flamm's paraboloid which is known in general relativity. 
    
    \item [(iv.)] For $k=2$, $R$ is a mylar balloon.
\end{itemize}
Thus, we get the curvatures of the loxodromes on such rotational surfaces in $\mathbb{E}^3$ by using the equations \eqref{kappaw} and \eqref{torw} with respect to the values of $k$.
\end{remark}

Now, we apply Euler's Theorem for loxodrome on rotational
surface $R$.

\begin{proposition}
\label{pro2}
Let $\alpha $ be a loxodrome on rotational surface $R$ whose
the principal curvatures are $\kappa _{1}$
and $\kappa _{2}$ with the principal
direction $\Phi_{s}=\gamma_{\theta }^{\prime }(s)$ 
and $\Phi_{\theta}$. Then, the normal curvature of $\alpha$ is obtained as: 
\begin{equation}
\label{norcurv}
\kappa _{n}=a^{2}\kappa _{1}+b^{2}\kappa _{2}
\end{equation}%
where $a=\cos \psi $, $b=\varepsilon \sin \psi $ and $\psi $ is an
angle between the loxodrome $\alpha $\ and the meridian curve $\gamma_{\theta }(s)$.
\end{proposition}
\begin{proof}
Assume that $\alpha$ is a loxodrome on the rotational surface $R$ in $\mathbb{E}^3$ given by \eqref{lox}. 
Then, we know the tangent vector $T$ of the loxodrome $\alpha$ 
given as \eqref{firstder}. 
%\begin{equation}
%T(t)=\left( af^{^{\prime }}\left( s\left( t\right) \right) \cos %\theta
%\left( t\right) -b\sin \theta \left( t\right) ,af^{\prime }\left( %s\left(
%t\right) \right) \sin \theta \left( t\right) +b\cos \theta \left( %t\right)
%,ag^{\prime }\left( s\left( t\right) \right) \right) .
%\end{equation}%
On the other hand, from \eqref{shapeop} and \eqref{normal}, we have
\begin{eqnarray}
\label{shapelox}
S(T)&=&-\frac{dN(\alpha(t) )}{dt} \\
&=&\left( ag^{\prime \prime }(s)\cos\theta(t)-b%
\frac{g^{\prime }}{f}\left( s\right) \sin \theta(t), 
ag^{\prime\prime }(s)\sin\theta(t)
+b\frac{g^{\prime }}{f}%
\left( s\right) \cos \theta (t),af^{\prime \prime }(s) \right) .  \notag
\end{eqnarray}%
By using \eqref{kappa_{1}}, \eqref{kappa_{2}}, \eqref{firstder} and \eqref{shapelox}, we get%
\begin{eqnarray*}
\kappa _{n}=II(T,T)=\left \langle S(T),T\right \rangle
=a^{2}\kappa _{1}+b^{2}\kappa _{2}.
\end{eqnarray*}
It completes the proof.
\end{proof}
\begin{corollary}
Let $R$ be a CRPC rotational surface in $\mathbb{E}^{3}$ such that $\kappa _{1}=-\tan ^{2}\psi\kappa _{2}$  
and $\alpha $ be a loxodrome on $R$ 
which cuts the meridian curves by the
angle $\psi $. Then, the loxodrome $\alpha $ is a asymptotic curve of $R$. 
\end{corollary}

\begin{proof}
It is obvious from Proposition \ref{pro2}.
\end{proof}

Using Theorem \ref{classminimal}, we give the following statement.
\begin{theorem}
Let $R$ be a minimal rotational surface in $\mathbb{E}^3$ 
and $\alpha $ be a loxodrome on $R$.
Then, the parametrization of $\alpha$ is given by 
\begin{equation}
\label{minimal}
\alpha \left( t\right) =\left( 
\begin{array}{c}
\frac{1}{n}\sqrt{1+n^{2}\left( at+p\right) ^{2}}\cos \left(\tan{\psi}
 \arcsinh \left( n\left(  at+p\right) \right)+q\right) , \\ 
\frac{1}{n}\sqrt{1+n^{2}\left( at+p\right) ^{2}}\sin \left(\tan{\psi}
 \arcsinh \left( n\left(  at+p\right) \right)+q\right) , \\ 
\frac{1}{n} \arcsinh \left( n\left( at+p\right) \right)+r 
\end{array}%
\right) 
\end{equation}
for constants $n,p, q, r$ and $a=\cos{\psi}$. 
\end{theorem}
\begin{proof}
Suppose that the rotational surface $R$ in $\mathbb{E}^3$ is minimal surface. 
Then, Theorem \ref{classminimal} implies that the functions $f(s)$ and $g(s)$ are given by \eqref{minimal1} and \eqref{minimal2}, respectively. 
Calculating $\theta(t)$ in Theorem \ref{thmloxodrome}  
for the functions $f(s)$ and $s(t)=at+c$, we obtain given result.
\end{proof}
\begin{theorem}
Let $R$ be a minimal rotational surface in $\mathbb{E}^3$.
Then, the curvatures of loxodrome $\alpha$ that cuts the
meridian curves of $R$ by the angle $\psi =\frac{\pi }{4}$ hold 
\begin{equation*}
\frac{\kappa }{\tau }=\lambda t+\mu 
\end{equation*}%
where $\lambda$ and $\mu$ are constants.
\end{theorem}

\begin{proof}
We assume that $R$ is a minimal surface and $\alpha $ is a loxodrome which
it cuts the meridian curves of $R$ by the angle $\psi =\frac{\pi }{4}.$ Then
by computing \eqref{kaploxeq} and \eqref{torloxeq}, we obtain the curvature and torsion of $\alpha $, respectively as follows:%
\begin{equation}
\label{kappamin}
\kappa
%=\pm\frac{\sqrt{f^2(s)-d^2}}{\sqrt{2}f^2(s)}
=\frac{1}{\sqrt{2}}\frac{f^{\prime }\left( s\right) }{f\left(
s\right)}
\end{equation}%
and 
\begin{equation}
\label{tormin}
\tau=
%\frac{\varepsilon d}{f^2(s)}
\frac{\varepsilon \sqrt{1-f^{\prime ^{2}}\left( s\right) }}{f\left(
s\right) }
\end{equation}%
in here $s=\frac{1}{\sqrt{2}}t+c$. 
On the other hand by using \eqref{minimal1}, \eqref{kappamin} and \eqref{tormin}, we obtain that 
\begin{equation*}
\frac{\kappa }{\tau }=\lambda t+\mu 
\end{equation*}%
where $\lambda =\frac{n}{2}$ and $\mu =\frac{(m+c)n}{\sqrt{2}}$ for 
constants $n$ and $m$. 
\end{proof}

\begin{corollary}
Let $R$ be a minimal rotational surface and 
$\alpha $ be a loxodrome that intersects
the meridian curves of $R$ at the angle $\psi =\frac{\pi }{4}$ for every points. Then, $\alpha$ is an asymptotic curve of $R.$
\end{corollary}

\begin{proof}
We suppose that the angle which is between the loxodrome and the meridian
curves is $\psi =\frac{\pi }{4}$ for every points and $R$ is a minimal
surface then from \eqref{norcurv}, we get $\kappa _{n}=0$ and it implies that $\alpha $
is an asymptotic curve of $R.$
\end{proof}

Now, we give an example for a loxodrome on a minimal rotational surface in $\mathbb{E}^3$.
\begin{example}

If we take $n=1$, $\psi=\frac{\pi}{4}$ and $\varepsilon=1$ in \eqref{minimal}. Then, the parametrization of loxodrome (asymptotic curve) is
\[\alpha(t)=\left(\sqrt{\frac{t^2}{2}+1}\cos\left(\arcsinh{\left(\frac{\sqrt{2}}{2}t\right)}\right), \sqrt{\frac{t^2}{2}+1}\sin\left(\arcsinh{\left(\frac{\sqrt{2}}{2}t\right)}\right), \arcsinh{\left(\frac{\sqrt{2}}{2}t\right)}\right).\]

The graphs of minimal rotational surface (catenoid) $R$, loxodrome (asymptotic curve) and meridian ($\theta=\pi/4$) can be drawn by using Mathematica as follows:

\begin{figure}[htbp]
\includegraphics[height=70mm]{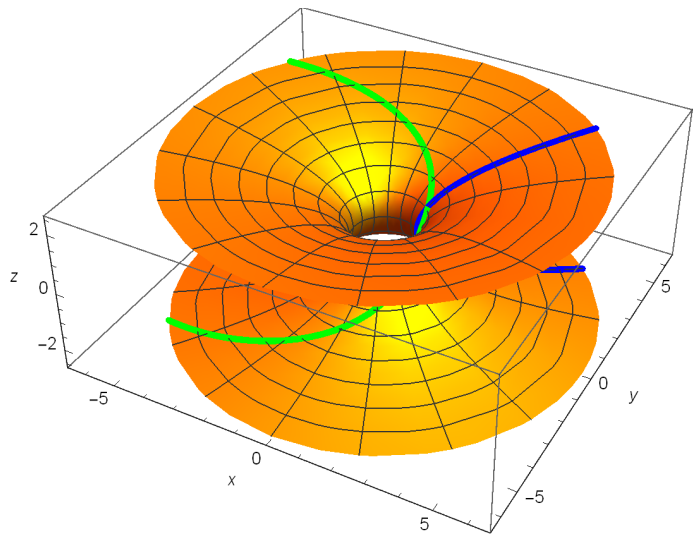}
\caption {Minimal rotational surface (catenoid), loxodrome (asymptotic curve) (green), meridian (blue).}
\label{fig4}
\end{figure}
\end{example}

\section{Conclusion}

In this paper, first we give the curvatures and torsions of loxodromes on rotational surfaces in Euclidean 3-space. Moreover, we obtain the parametrizations of loxodromes on some special rotational surfaces such as flat, constant Gaussian curvature, Weingarten and minimal. Also we give some important relations between loxodrome, helix and asymptotic curve. Finally, we introduce some examples of loxodromes on special rotational surfaces by using Wolfram Mathematica.

In the future, similarly we will study on the curvatures of space-like and time-like loxodromes on rotational surfaces respectively in Minkowski 3--space.

\end{document}